\documentclass[reqno,11pt]{amsart}

\usepackage{graphicx} 
\usepackage{amsmath,amssymb,amsthm,tikz}
\usepackage{xcolor}
\usepackage{esint}
\usepackage{booktabs}
\usepackage{scalerel}
\usepackage[usestackEOL]{stackengine}
\usepackage{mathrsfs}
\usepackage[normalem]{ulem}
\usepackage{hyperref}

\usepackage{pgfplots}
\usepackage{tikz}
\usepackage{pst-func}

\newtheorem{theorem}{Theorem}
\newtheorem{definition}{Definition}

\newtheorem{lemma}{Lemma}

\newtheorem{corollary}{Corollary}
\newtheorem{remark}{Remark}
\newtheorem{example}{Example}
\date{}
\numberwithin{equation}{section}
\numberwithin{theorem}{section}
\numberwithin{lemma}{section}
\numberwithin{corollary}{section}
\numberwithin{remark}{section} 
\numberwithin{proposition}{section}
\numberwithin{definition}{section}

\def\dashint{\,\ThisStyle{\ensurestackMath{
            \stackinset{c}{.2\LMpt}{c}{.5\LMpt}{\SavedStyle-}{\SavedStyle\phantom{\int}}}
        \setbox0=\hbox{$\SavedStyle\int\,$}\kern-\wd0}\int}

\def \R {\mathbb{R}}
\def \dist {\mathrm{dist}}

\def \loc {\mathrm{loc}}

\begin{document}
	
	\title[Interacting free boundaries in obstacle problems]{Interacting free boundaries\\ in obstacle problems}
	
	\author[D. J. Ara\'ujo]{Dami\~ao J. Ara\'ujo}
	\address{Department of Mathematics, Universidade Federal da Para\'iba, 58059-900, Jo\~ao Pessoa-PB, Brazil}{}
	\email{araujo@mat.ufpb.br}
	
	\author[R. Teymurazyan]{Rafayel Teymurazyan}
	\address{Applied Mathematics and Computational Sciences Program (AMCS), Compu\-ter, Electrical and Mathematical Sciences and Engineering Division (CEMSE), King Abdullah University of Science and Technology (KAUST), Thuwal, 23955-6900, Kingdom of Saudi Arabia and University of Coimbra, CMUC, Department of Mathematics, 3000-143 Coimbra, Portugal}{}
	\email{rafayel.teymurazyan@kaust.edu.sa}

\begin{abstract}
We study obstacle problems governed by two distinct types of diffusion operators involving interacting free boundaries. We obtain a somewhat surprising coupling property, leading to a comprehensive analysis of the free boundary. More precisely, we show that near regular points of a coordinate function, the free boundary is analytic, whereas singular points lie on a smooth manifold. Additionally, we prove that uncoupled free boundary points are singular, indicating that regular points lie exclusively on the coupled free boundary. Furthermore, optimal regularity, non-degeneracy, and lower dimensional Hausdorff measure estimates are obtained. Explicit examples illustrate the sharpness of assumptions.

\bigskip
		

\medskip

\end{abstract}

\bigskip
 
\maketitle
%\tableofcontents
\section{Introduction}\label{s1}

In recent years, the study of strongly coupled systems was boosted by applications in industry (catalysis processes), chemical engineering, and population dynamics (see \cite{AT23, CDV, CSS17, MNS21, MR20, ST18} and references therein). These models operate as systems of equations and free boundaries, in which a nonlinear diffusion process for the unknown (temperature of a given material) is observed only in regions where the other unknown (pressure) exceeds a certain threshold $\varphi$ (an obstacle), and conversely, a similar process is activated for the second unknown only in regions where the first one surpasses a given threshold $\psi$ (another obstacle). In financial mathematics, these types of problems are related to optimal switching when modeling switching of a state for cost reduction. Applications include stochastic switching zero-sum games and optimal stopping problems (see \cite{FS15} and references therein). The mathematical model can be formulated as an interactive obstacle-type problem
\begin{equation*}\label{motivation} 
\left\{
\begin{array}{rcll}
\min\{F(D^2u,Du,x),v-\varphi\} & = & 0 &\\[0.2cm]
\min\{G(D^2v,Dv,x),u-\psi\} & = & 0, &
\end{array}
\right. 
\end{equation*}
where $F$ and $G$ are diffusive elliptic operators.\\

We concentrate on problems ruled by two different types of diffusion operators -- the classical Laplacian and the infinity Laplacian -- a prototype of which is
\begin{equation*}\label{motivation2} 
\left\{
\begin{array}{rcll}
\min\{f - \Delta_\infty u,v-\mu\} & = & 0 & \mbox{ in } \,\, B_1\\[0.2cm]
\min\,\{\,g \,- \, \Delta v \,, u-\kappa\} & = & 0 & \mbox{ in } \,\, B_1,
\end{array}
\right. 
\end{equation*}
for constant thresholds $\mu,\kappa \in \mathbb{R}$ and $f,g\in L^\infty(B_1)$. In the non-variational context, the second-order operators mentioned above are defined by
$$ 
\Delta w:= trace(D^2 w) \quad \mbox{and} \quad \Delta_\infty w:= \langle D^2 w  \cdot Dw,D w \rangle.
$$ 
The Laplacian stands as the primary example of a diffusive second-order operator. For the related obstacle problem, the regularity of the solution and the corresponding free boundary is quite well understood \cite{C77, C98, CK80, CKS00, CS15, F18, FS14, PSU12, RO18}. On the other hand, the infinity Laplacian, characterized by its high elliptic degeneracy, has garnered substantial attention over the past three decades. It has strong connections with models describing scenarios such as random tug-of-war games and mass transfer problems and is linked to the best Lipschitz extension problem and the concept of comparison with cones, \cite{AU23, L16}. The study of obstacle problems governed by the infinity Laplacian was pioneered in \cite{RTU15}.\\

To keep the presentation of the ideas simple, we consider the following problem with zero obstacles
\begin{equation}\label{mainsys} 
\left\{
\begin{array}{rcll} 
\Delta_\infty u & \leq & f & \mbox{ in }\,\, B_1\\
\Delta v & \leq & g & \mbox{ in }\,\, B_1\\
\Delta_\infty u & = & f & \mbox{ in }\,\, B_1 \cap \{v>0\} \\
\Delta v & = & g & \mbox{ in }\,\, B_1\cap \{u>0\}.
\end{array}
\right. 
	\end{equation}
Unlike earlier works on obstacle problems with two or more equations (see, for example, \cite{CDV, GMR24, RT11, S05}), the system \eqref{mainsys} involves two quite different types of second-order operators, and equations are satisfied in two \textit{a priori} different unknown sets. 
	
We obtain regularity estimates for nonnegative solutions and corresponding free boundaries by providing qualitative properties for the \textit{coupled} free boundary
$$
\partial\{|(u,v)|>0\},    
$$
where 
    \begin{equation}\label{intrinsic}
        |(u,v)|:=u^{1/2}+v^{1/3}.
    \end{equation}
    The particular choice of exponents in \eqref{intrinsic} results from the intrinsic geometry of the problem. Solutions are understood in the viscosity sense, and the pair $(u,v)$ is said to be nonnegative if both $u$ and $v$ are nonnegative. 
    
    The existence of solutions is derived using Schaefer's fixed point theorem for an intrinsic penalized problem, as argued in \cite{AT23}. Furthermore, we obtain optimal growth and non-degeneracy estimates along the free boundary (Theorems \ref{growththm} and \ref{nondegthm} respectively) and conclude that the $(n-1)-$dimensional Hausdorff measure of the free boundary is locally finite (Theorem \ref{fbregularity}). Moreover, we show that an analog of Caffarelli's dichotomy holds. More precisely, we deduce that near regular points, the free boundary of the coordinate function $v$ is analytic, while singular free boundary points lie on a smooth manifold (Theorem \ref{dichotomy}). Additionally, we show that all points in the \textit{uncoupled} free boundary $\partial\{v>0\}\setminus\partial\{u>0\}$ are singular (Theorem \ref{singularset}).
    
    The main ingredient in the analysis is uncovering that, in fact, the free boundary of the coordinate solution $u$ coincides with that of an intrinsic combination of both coordinate solutions (Theorem \ref{coupledFBs}). This fact is a rather surprising result, as the study of the obstacle problem driven by the $\infty$-Laplacian is still in its infancy, with the only known regularity result being the one obtained in \cite{RTU15} (see also \cite{L14} for the blow-up limits). However, the estimate for viscosity solutions of $\Delta_\infty u\le f$ obtained in \cite{LW} enables equicontinuity for solutions of a built-in obstacle-type problem that paves the way to regularity theory. Furthermore, our results provide a new approach when studying the regularity of the free boundary in the obstacle problem driven by highly degenerate operators like the infinity Laplacian. In fact, if one can couple the problem with the one that solves \eqref{mainsys} for a suitable right-hand-side, then the free boundary of the obstacle problem ruled by the degenerate operator coincides with that of the coupled system, inheriting all the properties.\\

    The paper is organized as follows: in Section \ref{s2}, we prove the existence of viscosity solutions (Theorem \ref{existthm}). Section \ref{gensection} is devoted to the study of a built-in obstacle problem (Theorem \ref{growthonu}), which is then used in Section \ref{s5}, revealing a strong interplay between coordinate free boundaries and the intrinsic free boundary (Theorem \ref{coupledFBs}). Still, in Section \ref{s5}, we prove the regularity of solutions at points centered on the intrinsic free boundary (Theorem \ref{growththm}). We derive a weak comparison principle in Section \ref{s6} (Lemma \ref{WeakCP}), which then yields non-degeneracy of solutions (Theorem \ref{nondegthm}). Section \ref{FBregularity} is devoted to the regularity of the free boundary (Theorems \ref{fbregularity}-\ref{singularset}). In Section \ref{sMain}, through explicit examples, we illustrate the sharpness of assumptions in our main results. 

\medskip

\section{Existence of solutions}\label{s2}
	
In this section, we prove the existence of solutions for \eqref{mainsys} using Schaefer's fixed point theorem for an intrinsic penalized problem, as argued in \cite{AT23}, provided
\begin{equation}\label{bounded}
    f,g\in L^\infty(B_1).
\end{equation}
Solutions are understood in the viscosity sense. More precisely, for an open set $\mathcal{O}$ and an elliptic operator $H$, we say $w\in C(\mathcal{O})$ satisfies
$$
H(D^2 w)\le h\,\,(\ge h) \quad \mbox{in} \quad \mathcal{O},
$$
in the viscosity sense, if for any $\phi\in C^2(\mathcal{O})$ that touches $w$ from below (above) at $x_0\in\mathcal{O}$, one has $H(D^2\phi(x_0))\le h(x_0)\,\,(\ge h(x_0))$. In the viscosity sense, equation $H(D^2 w)=h$ means that the above inequalities hold simultaneously. 

To state the following lemma, we recall Schaefer's fixed point theorem (see, for example, \cite{Z86}).
\begin{theorem}[\textbf{Schaefer}]\label{Schaefer}
		If $X$ is a Banach space, $T:X\to X$ is continuous and compact, and the set
		$$
		\mathcal{E} = \{z\in X;\ \exists \ \theta\in [0,1]\ \text{such that}\ z=\theta T(z)\}
		$$
		is bounded, then $T$ has a fixed point.
	\end{theorem}
Let $\beta\in C^\infty(\R)$ be a non-decreasing function such that $\beta\in[0,1]$ and
	\begin{eqnarray*}\label{beta}
		\beta(s)=1 \mbox{ for } s \geq 1\quad \mbox{and} \quad \beta(s)=0 \mbox{ for } s \leq 0.
	\end{eqnarray*}
	For each $\varepsilon>0$, set 
	\begin{equation*}\label{Be}
		\beta_\varepsilon(s)= \beta(s/\varepsilon).
	\end{equation*}	 
 \begin{lemma}\label{Beta}
		If $\varphi,\psi\in C^{0,1}(\partial{B_1})$ and \eqref{bounded} holds, then there is a pair $(u_\varepsilon,v_\varepsilon)$ that is a viscosity solution of
		\begin{equation}\label{esys}
			\left\{
			\begin{array}{rcll}
				\Delta_\infty u_\varepsilon & = & f\,\beta_\varepsilon(v_\varepsilon) & \mbox{ in } B_1\\
				\Delta v_\varepsilon & = & g\,\beta_\varepsilon(u_\varepsilon) & \mbox{ in } B_1\\
				u_\varepsilon & = &\varphi & \mbox{ on } \partial B_1\\
				v_\varepsilon & = &\psi& \mbox{ on } \partial B_1. 
			\end{array}
			\right.
		\end{equation}		
	\end{lemma}
	\begin{proof}
		We follow the ideas of \cite[Proposition 2.1]{AT23} (see also \cite{RST17} and \cite{RTU19} for perturbation approach for the infinity Laplacian and Laplacian respectively). Let $\overline{u}$, $\overline{v}\in C^{0,1}(B_1)$, and define $T:C^{0,1}(\overline{B}_1)\times C^{0,1}(\overline{B}_1)\rightarrow
		C^{0,1}(\overline{B}_1)\times C^{0,1}(\overline{B}_1)$ by $T(\overline{u},\overline{v}):= (v,u)$, where $u$, $v$ are solutions of 
		\begin{equation}\label{Perron}
			\left\{
			\begin{array}{cccccc}
				\Delta_\infty u &=& f\,\beta_\varepsilon(\overline{v}) &\text{in}& B_1,\\
				u  &=&\varphi&\text{on} & \partial B_1
			\end{array}
			\right.
		\end{equation}
		and
		\begin{equation}\label{Perron2}
			\left\{
			\begin{array}{cccccc}
				\Delta  v &=& g\,\beta_\varepsilon(\overline{u}) &\text{in}& B_1\\			
				v & =& \psi&\text{on} & \partial B_1.
			\end{array}
			\right.
		\end{equation}	
		respectively (it is clear that such $u$ and $v$ depend on $\varepsilon$, but for the simplicity of notations, we drop the subscript $\varepsilon$ in the proof). Note that $T$ is well defined, as the classical Perron method guarantees the existence and uniqueness of $u$ and $v$. If $T$ has a fixed point, we are done. To apply Schaefer's theorem, we make sure its conditions are satisfied.
  
		\smallskip
  
		\textit{Step 1.} First we check that $T$ is continuous. Indeed, let
		$$
		(\overline{u}_k,\overline{v}_k)\rightarrow (\overline{u},\overline{v})\quad \text{in}\quad C^{0,1}(\overline{B}_1)\times C^{0,1}(\overline{B}_1).
		$$
		We aim to show that
		$$
		T(\overline{u}_k,\overline{v}_k)=T(\overline{u},\overline{v}).
		$$
		By the definition of $T$, we have 	
		$$
		T(\overline{u}_k,\overline{v}_k) = \left(v_k,u_k\right),
		$$
		where $u_k$ and $v_k$ are the \textit{unique} solutions of the corresponding problems \eqref{Perron} and \eqref{Perron2} respectively with $\overline{v}_k$ and $\overline{u}_k$ on the right hand side.
		Global Lipschitz estimates for \eqref{Perron} and \eqref{Perron2} (see \cite[Theorem 1.4]{SS14}) then imply existence of a universal constant $C>0$ such that 
		$$
		\|u_k\|_{C^{0,1}(\overline{B}_1)}\leq C\left(\|f\|_\infty  +\|\varphi\|_\infty\right),
		$$
		and
		$$
		\|v_k\|_{C^{0,1}(\overline{B}_1)}\leq C\left(\|g\|_\infty +\|\psi\|_\infty\right),
		$$
		since $\|\beta_\varepsilon\|_\infty=\|\beta\|_\infty\leq1$. Thus, $(u_k,v_k)$ is uniformly bounded. By the Arzel\'a-Ascoli theorem, up to a subsequence, it converges to some $(\tilde{u},\tilde{v})$. The stability of viscosity solutions under uniform limits then implies, as $k\to\infty$,
		$$
		T(\overline{u}_k,\overline{v}_k) =\left(v_k,u_k\right) \rightarrow (\tilde{v},\tilde{u}) = T(\overline{u},\overline{v}).
		$$

        \smallskip
  
		\textit{Step 2.} We then make sure that $T$ is compact. Indeed, if 
		$$
		(\overline{u}_k,\overline{v}_k)\in C^{0,1}(\overline{B}_1)\times C^{0,1}(\overline{B}_1)
		$$
		is a bounded sequence, then as above, 
		$$
		(v_k,u_k) = T(\overline{u}_k,\overline{v}_k)\in C^{0,1}(\overline{B}_1)\times C^{0,1}(\overline{B}_1)
		$$
		is bounded and, hence, has a convergent subsequence. Thus, $T$ is compact.

        \smallskip
		
		\textit{Step 3.} To use Theorem \ref{Schaefer}, it remains to see that the set of eigenvectors of $T$ is bounded, i.e., the set $\mathcal{E}$  with $X=C^{0,1}(\overline{B}_1)\times C^{0,1}(\overline{B}_1)$ is bounded. Note that $(0,0)\in\mathcal{E}$ if and only if $\theta=0$. Hence, we can assume that $\theta\neq0$. If $(\overline{u},\overline{v})\in \mathcal{E}$, then there exists $\theta\in (0,1]$ such that
		$$
		(\overline{u},\overline{v})=\theta T(\overline{u},\overline{v})=\theta(v,u),
		$$
		i.e., 
		$$
		\overline{u}=\theta v\quad\text{and}\quad \overline{v}=\theta u.
		$$
		Therefore,
		$$
		\left\{
		\begin{array}{ccccc}
			\Delta_\infty\,\overline{v}&=&\theta^3 f\,\beta_\varepsilon(v)&\mbox{in}&B_1, \\
			\overline{v}&=&\theta\varphi&\mbox{on}&\partial{B_1}
		\end{array}
		\right.
		$$
		and
		$$
		\left\{
		\begin{array}{ccccc}
			\Delta\,\overline{u}&=&\theta^2g\,\beta(\overline{u})&\mbox{in}&B_1,\\
			\overline{u}&=&\theta\psi& \mbox{on}&\partial{B_1}. \\
		\end{array}
		\right.
		$$
		Hence, as before,
		$$
		\|\overline{u}\|_{C^{0,1}(\overline{B}_1)}\leq C(\|g\|_\infty+\|\psi\|_\infty)
		$$
		and
		$$
		\|\overline{v}\|_{C^{0,1}(\overline{B}_1)}\leq C(\|f\|_\infty+\|\psi\|_\infty),
		$$
		where $C>0$ is a universal constant. Thus, $\mathcal{E}$ is bounded, and Theorem \ref{Schaefer} guarantees the existence of a fixed point, which completes the proof. 
	\end{proof}
	As a consequence, we obtain the existence of solutions for \eqref{mainsys}.
	\begin{theorem}\label{existthm}
		Assume that \eqref{bounded} holds, then \eqref{mainsys} has a viscosity solution $(u,v)$. Moreover, both coordinates $u$ and $v$ are Lipschitz continuous.
	\end{theorem}
	\begin{proof}
		Let $(u_\varepsilon,v_\varepsilon)$ be a viscosity solution of \eqref{esys}, which we know exists thanks to Lemma \ref{Beta}. Moreover, as we observe in the proof of Lemma \ref{Beta}, there exists a universal constant $C>0$, independent of $\varepsilon$, such that $\|u_\varepsilon\|_{C^{0,1}(\overline{B}_1)}<C$ and $\|v_\varepsilon\|_{C^{0,1}(\overline{B}_1)}<C$. By the Arzel\'a-Ascoli theorem, up to a subsequence, 
		$$
		u_\varepsilon\rightarrow u \quad\text{and}\quad v_\varepsilon\rightarrow v
		$$
		uniformly in $B_1$ for some $u$, $v\in C^{0,1}(B_1)$. It remains to check that $(u,v)$ is a solution of \eqref{mainsys}. We need to make sure the equations hold. If $z\in\{v>0\}$, then for $\varepsilon>0$ small enough one has
		$$
		v_\varepsilon(x) > \frac{v(z)}{4} \geq \varepsilon^2 \quad \mbox{for each } \; x \in B_r(z),
		$$ 
		where $r>0$ is small. Therefore, $\Delta_\infty u=f$ in $B_r(z)$. The other equation is checked similarly.
	\end{proof}

 \medskip
 
\section{A built-in obstacle problem}\label{gensection}

We emphasize that problem \eqref{mainsys} contains an obstacle-type problem governed by the infinity Laplacian:
\begin{equation}\label{domani}
\left\{
\begin{array}{rclcc}
\Delta_\infty u & \leq & f & \mbox{in} & B_1\, \\
u & \geq & 0 & \mbox{in} & B_1.
\end{array}
\right.
\end{equation}
Note that \eqref{domani} extends the class of free boundary problems treated in \cite{RTU15}, where it was shown that the solution of
$$
\min\{\Delta_\infty u - f, \, u \,\}=0,
$$
is $C^{1,1/3}$ at the free boundary. However, unlike the equation above, in \eqref{domani}, there is no viscosity sub-solution information in the region $\{u>0\}$. Nevertheless, we derive a growth estimate that paves the way to our main results. Before proceeding, we bring the following result from \cite[Lemma 2.2 (a)]{LW}, which enables equicontinuity property.	
	\begin{lemma}\label{Lipest}
		If $u$ is a viscosity solutions of $\Delta_\infty u\le f$ in $B_1$, where $f \in L^\infty(B_1)$, then it is locally Lipschitz continuous. Moreover, there exists $C>0$, depending only on $\|f\|_{\infty}$ and $n$, such that
		$$
		\sup\limits_{x,y \in B_{1/2}}\frac{|u(x)-u(y)|}{|x-y|}\le C\left(1+\|u\|_\infty\right).
		$$
	\end{lemma}	
    We use Lemma \ref{Lipest} to derive the following theorem.
	\begin{theorem}\label{growthonu}
		If $u$ is a viscosity solution of \eqref{domani}, and $f\in L^\infty(B_1)$, then there exists $C>0$, depending only on $\|u\|_\infty$, $\|f\|_\infty$ and $n$, such that for each $y\in\partial\{u>0\} \cap B_{1/2}$, one has
		\begin{equation}\label{growthest}
			u(x) \leq C|x-y|^{\frac{4}{3}},
		\end{equation}
		for any $|x-y|\leq 1/4$. Furthermore, $\partial\{u>0\} \subset \{|Du|=0\}$.
	\end{theorem}	
	\begin{proof}
		Without loss of generality, we may assume $y=0$. If \eqref{growthest} fails, then for each $k\in\mathbb{N}$, there exist $u_k$, a viscosity solution of \eqref{domani} and $r_k \in (0,1/4]$, such that $u_k(0)=0$ and
		\begin{equation*}\label{coffee}
			\sup_{B_{r_k}}u_k\ge k \,r_k^{\frac{4}{3}}.
		\end{equation*}		
		Set
		$$
		w_k(x):=\frac{u_k(r_k x)}{\displaystyle\sup_{B_{r_k}}u_k} \quad \mbox{in }\; B_1.
		$$
		Note that
		\begin{equation}\label{beergarden}
			\Delta_\infty w_k \leq \left(\frac{r_k^{\frac{4}{3}}}{\displaystyle\sup_{B_{r_k}}u_k}\right)^3\|f\|_\infty \leq \|f\|_\infty \quad \mbox{in }\; B_1.
		\end{equation}
		Additionally, $w_k(0)=0$, $w_k\ge0$ in $B_1$, and
		\begin{equation}\label{sup=1}
			\sup_{B_1}w_k=1.
		\end{equation}
		Applying Lemma \ref{Lipest} for \eqref{beergarden}, we conclude that $\{w_k\}_{k\in\mathbb{N}}$ is equicontinuous. Hence, since $w_k$ is bounded and up to a subsequence, it converges uniformly in $B_{1/2}$ to a function $w_\infty$. Letting $k\to\infty$ in \eqref{beergarden}, we deduce that $w_\infty$ is infinity super-harmonic in $B_{1}$. Therefore, by the strong maximum principle, \cite{BDL99}, it must vanish everywhere, which contradicts \eqref{sup=1}. Thus, \eqref{growthest} holds, which then implies
		$$
		D_i u(0)= \lim\limits_{h \to 0} \frac{u(he_i)-u(y)}{h}=\lim\limits_{h \to 0} \frac{u(he_i)}{h}\leq C\lim\limits_{h \to 0}h^{1/3}=0.
		$$
	\end{proof}

 \medskip
 
	\section{Coupling properties and regularity}\label{s5}
	
    In this section, we prove that there is a strong interplay between coordinate free boundaries $\partial\{u>0\}$, $\partial\{v>0\}$, and the intrinsic free boundary $\partial\{|(u,v)|>0\}$, where $|(u,v)|$ is defined by \eqref{intrinsic}. The idea is that the first coordinate function in the pair of solutions can be looked at as one acting freely, as observed in the previous section. The result is pivotal in establishing optimal interior regularity and growth estimates at the free boundary points.\\
    
    Observe that if in \eqref{mainsys} one has $f<0$ at some point, then $\Delta_\infty u <0$ in a small neighborhood of that point. This forces $u$ to be either strictly positive or identically zero in that neighborhood, and the second equation in \eqref{mainsys} becomes independent. Thus, there is no coupling. Similarly, if $g<0$ at a point, there is no interaction between equations. Therefore, to ensure that we are dealing with coupled equations, hereafter, we assume that
    \begin{equation}\label{nondegRHS}
    \inf_{B_1}\,\min\{f(x),g(x)\} \geq c_0,
    \end{equation} 
    for some $c_0>0$. Assumption \eqref{nondegRHS} is vital for our analysis (see Example \ref{example1}).   

    \smallskip

    Below is the main result of this section.
	\begin{theorem}\label{coupledFBs} 
		Assume $(u,v)\ge0$ is a viscosity solution of \eqref{mainsys} and \eqref{bounded}, \eqref{nondegRHS} hold, then 
		\begin{equation}\label{set2}
			\partial\{u > 0\} \subset \partial\{v > 0\} \subset \overline{\{u>0\}}.
		\end{equation}
		Furthermore,
		\begin{equation}\label{set3}
			\{v > 0\} \subset \{u > 0\} \quad \mbox{and} \quad 
			\partial\{|(u,v)| > 0\}=\partial\{u > 0\}.
		\end{equation}
	\end{theorem}
\begin{proof}
		We divide the proof into three steps:

        \smallskip
		
		\textit{Step 1.} First, we observe that
		\begin{equation}\label{set1}
			int\{u=0\} = int\{v=0\}.
		\end{equation}
		In fact, since $f>0$ in $B_1$, if $int\{u=0\} \neq \emptyset$, then for any small ball $B\subset int\{u=0\}$, we have $B\cap\{v>0\}=\emptyset$, since otherwise one would have $\Delta_\infty u=f$ at points where $u\equiv0$, which is not possible, as $f>0$. Hence, $int\{u=0\}\subset int\{v=0\}$. Analogously, as $g>0$ in $B_1$, we obtain $int\{v=0\}\subset int\{v=0\}$. 

        \smallskip
		
		\textit{Step 2.} Next, we show that \eqref{set2} holds. Indeed, thanks to Theorem \ref{growthonu}, $\partial\{u>0\}\subset\{|Du|=0\}$. Therefore, for any $\phi \in C^2$ touching $u$ from above at $y\in \partial\{u>0\}$, we have 
		$$
		D\phi(y)=Du(y)=0.$$ 
		Thus, $\Delta_\infty \phi(y)=0$ and so, 
		$$
		\Delta_\infty u(y)\le0
		$$
		in the viscosity sense. As $f>0$, then
        $$
        \Delta_\infty u(y)<f(y).
        $$
        On the other hand, in \eqref{mainsys} for points in $B_1\cap\{v>0\}$ we have
        $$
        \Delta_\infty u=f.
        $$
        Therefore,         
		$$
		\partial\{u>0\} \cap \{v>0\} = \emptyset.
		$$
		Then from \eqref{set1} we deduce 
		$$
		\partial\{u>0\}\cap int\{v=0\}=\partial\{u>0\}\cap int\{u=0\}=\emptyset,
		$$
		hence, $\partial\{u>0\}\subset\partial\{v>0\}$. Similarly, 
		$$
		\partial\{v>0\} \cap int\{u=0\} = \partial\{v>0\} \cap int\{v=0\} = \emptyset,
		$$
		and so, $\partial\{v>0\} \subset \overline{\{u>0\}}$, and \eqref{set2} is proven.

        \smallskip
		
		\textit{Step 3.} In the final step, we prove \eqref{set3}. Let $x\in\{v>0\}$. From \eqref{set1} we know that $x\in\overline{\{u>0\}}$. The latter with \eqref{set2} implies that $x\in \{u>0\}$. This proves the first part of \eqref{set3}. To check the second part, note that if $x\in\partial\{|(u,v)|>0\}$, then $u(x)=v(x)=0$, which, combined with \eqref{set1} and \eqref{set2} leads to 
		$$
		x\in\partial\{u>0\}\cap\partial\{v>0\}=\partial\{u>0\},
        $$
        i.e., $\partial\{|(u,v)|>0\}\subset\partial\{u>0\}$. On the other hand, if $x \in \partial \{u>0\}$, then recalling once more \eqref{set2}, we have $x\in \partial \{v>0\}$, and so
        $|(u(x),v(x))|=0$. Observe that $x\notin int\{|(u,v)|=0\}$, since otherwise it would imply that $x\in int\{u=0\}$. Thus, $x \in \partial\{|(u,v)|>0\}$, i.e., $\partial\{u>0\}\subset\partial\{|(u,v)|>0\}$.
	\end{proof}
    \begin{remark}\label{strictinclusion}
        In general, the inclusion in \eqref{set2} is strict (see Example \ref{sinclusion} below), i.e., there are points in $\partial\{v>0\}$ that are not in $\partial\{u>0\}$ (see Figure \ref{fig:1} below). Additionally, the equality in \eqref{set3} yields the following decomposition
        $$
        \partial\{v>0\}=\left(\partial\{v>0\}\setminus\partial\{u>0\}\right)\cap\partial\{|(u,v)|>0\}.
        $$
    \end{remark}
    Theorem \ref{coupledFBs} implies optimal growth control for solutions along the coupled free boundary points, as stated in the following result.   
	\begin{theorem}\label{growththm}
		If $(u,v)\ge0$ is a viscosity solution of \eqref{mainsys}, and \eqref{bounded}, \eqref{nondegRHS} hold, then $(u,v)$ is locally of class $C^{0,1}(B_1)\times C^{1,1}(B_1)$. Moreover, there exist positive constants $C$ and $r_0$, depending only on $n$,  $\|f\|_\infty$, $\|g\|_{\infty}$, $\|u\|_\infty$ and $\|v\|_\infty$, such that for $y\in\partial\{|(u,v)|> 0\} \cap B_{1/2}$, there holds    
		\begin{equation}\label{growthest||}
			\sup\limits_{B_r(y)}|(u,v)|\leq Cr^{\frac{2}{3}},
		\end{equation}
		for each $0<r\leq r_0$.
	\end{theorem}	
	\begin{proof}
		Unlike similar results in \cite{ALT16, AT23, DT20}, our proof does not employ a flattening argument and instead makes use of Lemma \ref{Lipest} and Theorem \ref{growthonu}.
		Observe that from \eqref{set1}, \eqref{set2} and \eqref{set3}, we derive 
		\begin{equation*}
			\begin{array}{rcl}
				B_1 \cap \{u>0\} & = & B_1 \setminus (\overline{\{u=0\}}) \\[0.2cm]
				& = & B_1 \setminus ({\partial \{u>0\} \cup int\{v=0\}}) \\[0.2cm]
				& \supseteq & B_1 \setminus ({\partial \{v>0\} \cup int\{v=0\}})
				\\[0.2cm]
				& = & B_1 \cap \{v>0\}.
			\end{array}
		\end{equation*}
		Hence, $v$ solves the following classical obstacle problem
		\begin{equation}\label{obstacleproblem}
			\left\{
			\begin{array}{rcll}
				v & \geq & 0 & \mbox{ in } B_1\\
				\Delta v & \leq & g & \mbox{ in } B_1\\
				\Delta v & = & g & \mbox{ in } B_1\cap \{v>0\}. \\
			\end{array}
			\right. 
		\end{equation}
		Therefore (see, for example, \cite{C98}), $v\in C_{\loc}^{1,1}(B_1)$, which combined with Lemma \ref{Lipest}, yields that $(u,v)$ is locally $C^{0,1}\times C^{1,1}$. In particular, for $y\in\partial\{v>0\}\cap B_{1/2}$, one has
		$$
		\sup\limits_{B_r(y)}v\leq Cr^2.
		$$
		The latter, with Theorem \ref{growthonu} and Theorem \ref{coupledFBs}, gives \eqref{growthest||}.
	\end{proof}
 
 \medskip
    
	\section{Non-degeneracy and consequences}\label{s6}
		
	In this section, we prove a comparison principle and derive non-degeneracy, positive density, and porosity results for solutions of \eqref{mainsys}. To proceed, we fix an order when comparing the pairs. Namely, we say $(a,b)<(c,d)$, if $a<c$ and $b<d$. Inequalities $(a,b)>(c,d)$, $(a,b)\leq(c,d)$ and $(a,b)\geq(c,d)$ are understood analogously.
	\begin{lemma}\label{WeakCP}
		Let $(u_i,v_i)$ be a non-negative viscosity solution of
		\begin{equation}\label{1sys} 
			\left\{
			\begin{array}{rcll}
				\Delta_\infty u_i & \leq & f_i & \mbox{ in }\, B_1\\[0.05cm]
				\Delta v_i & \leq & g_i & \mbox{ in }\, B_1\\[0.05cm]
				\Delta_\infty u_i & = & f_i & \mbox{ in }\, B_1 \cap \{v>0\} \\[0.05cm]
				\Delta v_i & = & g_i & \mbox{ in }\, B_1\cap \{u>0\}, \\[0.05cm]
			\end{array}
			\right. 
		\end{equation} 
		with $(f_i,g_i)\in C(B_1) \times C(B_1)$, $i=1,2$ and $g_1<g_2$ in $B_1$. If
		$$
		v_1\geq v_2 \quad \mbox{on } \; \partial B_1,
		\quad
		\mbox{then}  
		\quad 
		v_1\geq v_2\quad \mbox{in }\;B_1.
		$$
	\end{lemma}    
	\begin{proof}
		Indeed, if        
		$$
		\mathcal{O}:=\{x\in B_1;\,\,\,v_2(x)>v_1(x)\}\neq\emptyset,
		$$
		then, since $v_1\geq 0$, one has $\mathcal{O}\subset\{v_2>0\}\cap B_1$, which, recalling \eqref{set3} and \eqref{1sys}, leads to
		\begin{equation}\nonumber
			\left\{
			\begin{array}{rcll}
				\Delta v_1 & \leq & g_1 & \mbox{ in } \mathcal{O}\\[0.05cm]
				\Delta v_2 & = & g_2 & \mbox{ in } \mathcal{O} \\[0.05cm]
				v_1 & = & v_2 & \mbox{ on } \partial\mathcal{O}. \\[0.05cm]
			\end{array}
			\right. 
		\end{equation} 
		The comparison principle then gives
		\begin{equation}\nonumber
			v_2\leq v_1 \quad\text{in}\quad \mathcal{O},
		\end{equation}
		which contradicts to the definition of $\mathcal{O}$. Thus, $\mathcal{O}=\emptyset$.
	\end{proof}    
To prove the non-degeneracy of solutions, we assume that
    \begin{equation}\label{continuousRHS}
        f,g\in C(B_1)\cap L^\infty(B_1).
    \end{equation}
	\begin{theorem}\label{nondegthm}
		Assume $(u,v)\ge0$ is a viscosity solution of \eqref{mainsys}, and \eqref{nondegRHS}, \eqref{continuousRHS} hold. There exists a constant $c>0$, depending only on $c_0$, such that for $y\in\overline{\{|(u,v)|>0\}}$, there holds
		\begin{equation*}\label{nondegest}
			\sup_{\overline{B}_r(y)}|(u,v)|\geq c \,r^{\frac{2}{3}},
		\end{equation*}
		for each $r\in\left(0,1/2\right]$.
	\end{theorem}	
	\begin{proof}
By \eqref{set1}, $\{|(u,v)|>0\} \subset \overline{\{v>0\}}$. Hence, $\overline{\{|(u,v)|>0\}} = \overline{\{v>0\}}$. By continuity, it is enough to show the estimate for $y\in\{v>0\}\cap B_{1/2}$. With no loss of generality, we may assume $y=0$. Set
		$$
		\overline{u}(x):= C_f \,\dfrac{81}{64}|x|^{\frac{4}{3}}  \quad \mbox{and} \quad \overline{v}(x):=C_g\, \frac{1}{2}|x|^2,
		$$  
		for constants 
		$$
		C_f:=\left(\frac{\displaystyle\inf_{B_1} f}{2}\right)^{\frac{1}{3}} \quad \mbox{and} \quad C_g:=\displaystyle\frac{1}{2}\inf_{B_1}g.
		$$		
        Note that $(\overline u, \overline v)$ is a nonnegative viscosity solution of \eqref{mainsys}, with source terms which are strictly less than $(f,g)$. Now if
		\begin{equation}\label{5.3}
			v<\overline{v}\quad \mbox{in } \; \partial{B_r},
		\end{equation} 
		then Lemma \ref{WeakCP} provides 
		$$
		v\le\overline{v} \quad \mbox{in } \; B_r.
		$$
		In particular, $0<v(0)\le\overline{v}(0)$, which contradicts to $\overline{v}(0)=0$. Thus, \eqref{5.3} does not hold, therefore, there exists a point $x\in \partial{B_r}$ such that
		\begin{equation*}\label{nondegproof}
			v(x)\ge\overline{v}(x).
		\end{equation*}
		We then estimate
		\begin{equation}\nonumber
			\sup_{\overline{B}_r}|(u,v)|\geq \sup_{\partial B_r}v^{\frac{1}{3}} \ge\overline{v}^{\frac{1}{3}}(x)=\frac{C_g}{2}\, r^{\frac{2}{3}}. 
		\end{equation}
	\end{proof} 
	As a consequence, we get the following positive density result.
	\begin{corollary}\label{c5.1}
		Assume $(u,v)\ge0$ is a viscosity solution of \eqref{mainsys}, and \eqref{nondegRHS}, \eqref{continuousRHS} hold. There exists $c>0$, such that for $x_0\in \partial\{|(u,v)|>0\}\cap B_\frac{1}{2}$, one has
		$$
		|B_r(x_0)\cap\{|(u,v)|>0\}|\geq cr^n,
		$$
		for each $r\in \left(0,\frac{1}{2}\right)$.
	\end{corollary}
	\begin{proof}
		Theorem \ref{nondegthm} guarantees existence of a point $y\in\{ \overline{|(u,v)>0}\}$ such that
		\begin{equation}\label{Density_1}
			|(u(y),v(y)|=\sup_{B_r(y)}|(u,v)|\geq cr^\frac{2}{3}.
		\end{equation}
		The idea now is to make sure that we can choose $0<\tau<1$ such that
		$$
		B_{\tau r}(y)\subset\{|(u,v)|>0\}.
		$$
		Set $d(x):=|x-y|$. Using Theorem \ref{growththm} and \eqref{Density_1}, we obtain
		$$
		cr^\frac{2}{3}\leq|(u(y),v(y))|\leq C(d(x)+|(u(x),v(x))|^\frac{3}{2})^\frac{2}{3},
		$$
		that is,
		$$
		|(u(x),v(x))|\geq \left(\left(\frac{c}{C}\right)^\frac{3}{2}r-d(x)\right)^\frac{2}{3}.
		$$
		Thus, if $\tau<\left(\frac{c}{C}\right)^\frac{3}{2}$, then $|(u(x),v(x))|>0$, i.e., $x\in \{|(u,v)|>0\}$. Therefore,
		$$
		|B_r(x_0)\cap\{|(u,v)|>0\}|\geq |B_r(x_0)\cap B_{\tau r}(y)|\geq cr^n,
		$$
        which is the desired result.
	\end{proof}
	\begin{corollary}
		If $(u,v)\ge0$ is a viscosity solution of \eqref{mainsys}, and \eqref{nondegRHS}, \eqref{continuousRHS} hold, then there exists a universal constant $c>0$ such that 
		$$
		\fint_{B_r(x_0)}|(u,v)|\,dx\geq cr^\frac{2}{3},
		$$
		for all $x_0\in \partial\{|(u,v)|>0\}\cap B_{\frac{1}{2}}$ and $\rho\in \left(0,\frac{1}{2}\right)$.
	\end{corollary}	
	\begin{proof}
		As in the proof of Corollary \ref{c5.1},
		$$
		B_{\tau r}(y)\subset\{|(u,v)|>0\},
		$$
		and \eqref{Density_1} holds, therefore
		$$
		\fint_{B_r(x_0)}|(u,v)|\,dx\ge
		\fint_{B_r(x_0)\cap B_{\tau r}(y)}|(u,v)|\,dx\ge cr^\frac{2}{3},
		$$		
	\end{proof}
	To state the next consequence, we define porous sets.
	\begin{definition}
		A set $E\subset\R^n$ is called porous with porosity constant $\delta>0$, if there is $\rho>0$ such that for each $x\in E$ and $r\in(0,\rho)$ there is $y\in\R^n$ such that $B_{\delta r}(y)\subset B_r(x)\setminus E$.
	\end{definition}
	The Hausdorff dimension of a porous set does not exceed $n-C\delta^n$, where $C>0$ is a constant depending only on $n$ (see, for example, \cite{MV87}). Hence, the Lebesgue measure of a porous set is zero.
		\begin{corollary}\label{porosity}
		If $(u,v)\ge0$ is a viscosity solution of \eqref{mainsys}, and \eqref{nondegRHS}, \eqref{continuousRHS} hold, then the coupled free boundary is porous, and therefore, has Lebesgue measure zero.
	\end{corollary}
	\begin{proof}
		Let $x\in E:=\partial\{|(u,v)|>0\}\cap\overline{B_r(x_0)}$, for $x_0\in B_1$ such that $\overline{B_{2r}(x_0)}\subset B_1$. For each $\tilde{r}\in(0,r)$, we have $\overline{B_{\tilde{r}}(x)}\subset B_{2r}(x_0)\subset B_1$. From Theorem  \ref{nondegthm}, there exists $y\in\partial B_r(x)$ such that 
		$$
		|(u(y),v(y))|\ge cr^{\frac{2}{3}},
		$$
		for a constant $c>0$. Hence, $y\in B_{2r}(x_0)\cap\{(u,v)>0\}$. Set $d(y):=\dist(y,\overline{B_{2r}(x_0)}\cap\{(u,v)=0\})$, then Theorem \ref{growththm} provides
		$$
		|(u(y),v(y))|\le C\left[d(y)\right]^{\frac{2}{3}},
		$$
		for a constant $C>0$. Therefore, setting 
		$$
		\delta:=\min\left\{\frac{1}{2},\left[cC^{-1}\right]^{\frac{3}{2}}\right\}<1,
		$$ 
		we have
		$$
		d(y)\ge\delta r.
		$$
		Hence, $B_{\delta r}(y)\subset B_{d(y)}(y)\subset\{(u,v)>0\}$. In particular,  
		$$
		B_{\delta r}(y)\cap B_r(x)\subset\{(u,v)>0\}.
		$$
		On the other hand, if $z\in[x,y]$ is such that $|z-y|=\delta r/2$, then
		$$
		B_{(\delta/2)r}(z)\subset B_{\delta r}(y)\cap B_r(x).
		$$		
		Thus,
		$$
		B_{(\delta/2)r}(z)\subset B_{\delta r}(y)\cap B_r(x)\subset B_r(x)\setminus\partial\{u>0\}\subset B_r(x)\setminus E,
		$$
		i.e., $E$ is porous with porosity constant $\delta/2$. 
	\end{proof}

 \smallskip

 \section{Free boundary regularity}\label{FBregularity}

    From the previous section, we already know that the Lebesgue measure of the coupled free boundary $\partial\{|(u,v)|>0\}$ is zero, Corollary \ref{porosity}. In this section, we conclude that its $(n-1)$ dimensional Hausdorff measure is finite, deducing that up to a negligible set of null perimeter, the free boundary is a union of at most countable number of $C^1$ hyper-surfaces. Additionally, we show that Caffarelli's dichotomy holds in the sense that any point in $\partial\{|(u,v)|\geq 0\}$ is either a regular free boundary point for the coordinate function $v$, and around that point $\partial\{v>0\}$ is analytic, or the point is singular, and the set of singular points lies on a $C^1$-manifold. Furthermore, by using blow-up analysis, we conclude that all the points in the uncoupled free boundary $\partial\{v>0\}\setminus\partial\{u>0\}$ are singular.
    \begin{theorem}\label{fbregularity}
        If $(u,v)\ge0$ is a viscosity solution of \eqref{mainsys}, \eqref{nondegRHS} holds and $f\in L^\infty(B_1)$, $g\in C_{\loc}^{0,1}(B_1)\cap L^\infty(B_1)$, then the $(n-1)$-dimensional Hausdorff measure of the free boundary $\partial\{|(u,v)|>0\}$ is locally finite. 
    \end{theorem}
    \begin{proof}
        Indeed, \eqref{obstacleproblem} reveals that $v$ is the solution of the classical obstacle problem, therefore the $(n-1)$-dimensional Hausdorff measure of the free boundary $\partial\{v>0\}$ is locally finite, \cite[Corollary 4]{C98}, \cite[Theorem 3.3]{LS03}. Note now that from Theorem \ref{coupledFBs} we have $\partial\{|(u,v)|>0\}\subset\partial\{v>0\}$. 
    \end{proof}
    \begin{remark}
        As observed in \cite[Theorem 4.2]{CLRT14}, we can replace Lipschitz continuity assumption on $g$ by assuming the following integrability condition
        $$
        \int_{B_r}|\nabla g|\,dx\le C_0r^{n-1},\,\,\,\forall\; r\in(0,3/4),
        $$
        for some universal constant $C_0>0$.
    \end{remark}
    \begin{remark}
    Since $\partial\{|(u,v)|>0\}$ has locally finite $(n-1)$-dimensional Hausdorff measure, the set $\{|(u,v)|>0\}$ has locally finite perimeter in $B_1$, \cite{EG15}.
    \end{remark}
    \begin{remark}
    Up to a negligible set of null perimeter, the free boundary $\partial\{|(u,v)|>0\}$ is a union of, at most, a countable family of $C^1$ hyper-surfaces, \cite{G84}.
    \end{remark}

    Next, we classify free boundary points and deduce free boundary regularity for one of the coordinate functions. For simplicity of the argument, we will assume 
    \begin{equation*}\label{simplicity}
        f\equiv g\equiv1,
    \end{equation*}
    i.e., $(u,v)$ is a viscosity solution of
    \begin{equation}\label{mainsys2} 
\left\{
\begin{array}{rcll} 
\Delta_\infty u & \leq & 1 & \mbox{ in }\,\, B_1\\
\Delta v & \leq & 1 & \mbox{ in }\,\, B_1\\
\Delta_\infty u & = & 1 & \mbox{ in }\,\, B_1 \cap \{v>0\} \\
\Delta v & = & 1 & \mbox{ in }\,\, B_1\cap \{u>0\}.
\end{array}
\right. 
	\end{equation}
    To classify free boundary points, we recall the following definition.
    \begin{definition}\label{regularpoint}
        A free boundary point $x_0\in\partial\{v>0\}$ is called 
        \begin{itemize}
            \item regular, if up to a sub-sequence, as $r\to0^+$, one has
        $$
        \frac{v(x_0+rx)}{r^2}\to\frac{1}{2}[(e\cdot x)_+]^2,
        $$
        for some unit vector $e\in\mathbb{S}^{n-1}$;
            \item  singular, if up to a sub-sequence, as $r\to0^+$, one has
        $$
        \frac{v(x_0+rx)}{r^2}\to\frac{1}{2}\langle Ax,x\rangle,
        $$
        for some non-negative definite matrix $A\in\R^{n\times n}$ with $tr(A)=1$.
        \end{itemize}
    \end{definition}     
    Since $v\ge0$ solves the classical obstacle problem \eqref{obstacleproblem}, the result below follows as a direct consequence of the celebrated Caffarelli's dichotomy, see \cite{C77}, \cite[Theorem 8]{C98}, \cite[Theorem 7.3]{F18}.
    \begin{theorem}\label{dichotomy}
        If $(u,v)\ge0$ is a viscosity solution of \eqref{mainsys2}, and $x_0\in\partial\{v>0\}$, then either $x_0$ is a regular point, and in its small neighborhood the free boundary $\partial\{v>0\}$ is an analytic hyper-surface consisting only of regular points, or $x_0$ is a singular point, and all singular points lie in a $k$-dimensional $C^1$ manifold, where $k$ is the dimension of the kernel of the matrix $A$.
    \end{theorem}
    Furthermore, we show that all points on the uncoupled free boundary are singular (recall Remark \ref{strictinclusion}).     
    \begin{theorem}\label{singularset}
     If $(u,v)\ge0$ is a viscosity solution of \eqref{mainsys2}, then singular points exhaust the set $\partial\{v>0\}\setminus\partial\{u>0\}$.
    \end{theorem}
    \begin{proof}
    We argue by contradiction and assume that there is a regular point $x_0\in\partial\{v>0\}$ that is in $\partial\{v>0\}\setminus \partial\{u>0\}$. Without loss of generality, we may assume that $x_0=0$. Theorem \ref{coupledFBs} then provides
    \begin{equation}\label{contradictoryassumption}
        0\in \partial\{v>0\} \cap \{u>0\}.
    \end{equation}
    For each $r>0$, we define
    \begin{equation*}
    v_r(x):=\frac{v(rx)-v(0)}{r^2}    
    \end{equation*}
    and observe that
    $$
    v_r(0)=0 \quad \mbox{and} \quad  \sup_{B_1}v_r\le C
    $$
    where the constant $C>0$ depends only on $\|u\|_\infty$ and $\|v\|_\infty$.
    Ascoli-Arzel\`a theorem then implies that the family $\{v_r\}_r$ is compact in the $C^{0,1} \times C^{1,1}$ topology. Therefore, applying Theorem \ref{dichotomy}, we arrive at
    $$
    v_r\to v_\infty,\,\mbox{ as }\; r\to 0^+,
    $$
    where     
    \begin{equation}\label{blow}
    v_\infty(x)=\frac{1}{2}[(e\cdot x)_+]^2,
    \end{equation}
    for some $e\in\mathbb{S}^{n-1}$. From \eqref{contradictoryassumption}, we have $u>0$ in a small ball $B_\tau$ centered at the origin. Hence, $\Delta v=1$ in $B_\tau$ and so 
    $$
    \Delta v_r = 1\,\,\mbox{ in }\,\, B_{\tau/r}.
    $$
    The latter implies
    $$
    \Delta v_\infty = 1\,\,\mbox{ in }\,\,\mathbb{R}^n,
    $$
    which contradicts \eqref{blow}.
    \end{proof}

%%%%%%%%%%%%%%%%%%%%%%%%%%%%%%%%%
\begin{figure}[h!]
\begin{tikzpicture}[scale=0.55,cap=round]
 \tikzset{axes/.style={}}
%graphic
%cusp1
\begin{scope}[style=axes]
\draw[thick, gray!80, cm={cos(180) ,-sin(180) ,sin(180) ,cos(180) ,(3.0,5.0)}] 
(0.25,0.4) to [out=10,in=-90]  (1.5,2.0)
 to [out=-90, in=175] (2.75,0.4);
%cusp2
\draw[thick, gray!80] (0.25,0.4) to [out=10,in=-90]  (1.5,2.0) to [out=-90, in=175] (2.75,0.4);
%text
\node at (1.5,2.45) {$u>0$};
\draw[thin, black!80, <-] (1.7,3.5) -- ++ (2.3,0.8) node[above right] at (4.0,3.6) {$\partial\{u>0\}$};
%blue circle
\draw[blue!80] (1.5,2.5) circle (2.45cm);
%%%
\end{scope}
\end{tikzpicture}
%%%SEGUNDA
\begin{tikzpicture}[scale=0.55,cap=round]
 \tikzset{axes/.style={}}
%line
\draw[thick, black] (1.5,2.0) -- (1.5,3.1);
\node at (0.3,2.45) {$v>0$};
\node at (2.8,2.45) {$v>0$};
%cusp1
\begin{scope}[style=axes]
\draw[thick, gray!80, cm={cos(180) ,-sin(180) ,sin(180) ,cos(180) ,(3.0,5.0)}] (0.25,0.4) to [out=10,in=-90]  (1.5,2.0)
 to [out=-90, in=175] (2.75,0.4);
%cusp2
\draw[thick, gray!80] (0.25,0.4) to [out=10,in=-90]  (1.5,2.0)
 to [out=-90, in=175] (2.75,0.4);
\draw[thin, black!80, <-] (1.6,2.55) -- ++ (2.3,1.5) node[above right] at (4.0,3.6) {$\partial\{v>0\}\setminus \partial\{u>0\}$};
%blue circle
\draw[blue!80] (1.5,2.5) circle (2.45cm);
%%%
\end{scope}
\end{tikzpicture}
\caption{An illustration of uncoupled free boundary points shaped in a singular fashion. All regular free boundary points are in the coupled free boundary.} 
\label{fig:1}
\end{figure}
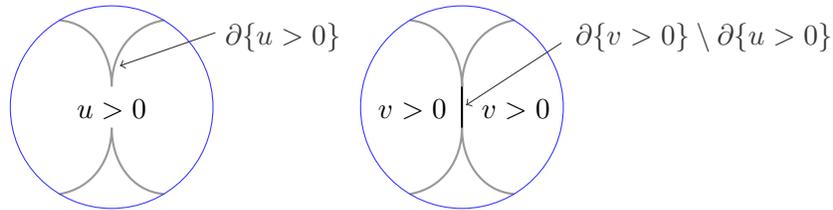
%%%%%%%%%%%%%%%%%%%%%%%%%%%%
    
\medskip
    
 \section{Examples and beyond}\label{sMain}

In the final section, we bring explicit examples emphasizing the sharpness of assumptions in our main results. The first example shows that assumption \eqref{nondegRHS} is vital for our analysis.
\begin{example}\label{example1}
For a given $\alpha>0$ and $\epsilon>0$ (small), take a constant $C_\alpha>0$, such that the pair $(u,v_\epsilon)$, for  
	\begin{equation*}\label{ex1}
		\begin{array}{cc}
			u(x):=\dfrac{81}{64}(x_1)_+^{4/3} \quad \mbox{and} \quad v_\epsilon(x):=\left(x_1-\epsilon\right)_+^{2+\alpha},
		\end{array}
	\end{equation*}
	solves \eqref{mainsys} with $f=1$ and $g_\epsilon=C_\alpha\left(x_1-\epsilon\right)_+^{\alpha}$. Observe that
	$$
	\displaystyle\inf_{B_1} g_\epsilon =0,
	$$
	i.e. \eqref{nondegRHS} fails, and also	
	$$
	\partial\{u>0\}= \{x_1=0\} \quad \mbox{and} \quad \partial\{v_\epsilon>0\}= \{x_1=\epsilon\},
	$$
    therefore, $\partial\{|(u,v_\epsilon)|>0\}=\emptyset$. Hence, the lack of condition \eqref{nondegRHS} leads to the failure of \eqref{set2} and \eqref{set3}.
    \end{example}	
    The next example highlights the importance of assumption $(u,v)\geq0$.
    \begin{example}
	Unlike classical obstacle problems (see, for example, \cite{C98, CLRT14, RT11}), solutions of \eqref{mainsys} may fail to be nonnegative even when the boundary data is nonnegative, as shows the following example. Indeed, the pair of functions
	\begin{equation*}\label{ex2}
		u(x):=\dfrac{81}{64}|x|^{4/3} \quad \mbox{and} \quad  v_\epsilon(x):=\frac{1}{2}|x|^2-\epsilon^2, \quad x \in B_1,
	\end{equation*}
	solves (for $\epsilon>0$ small) \eqref{mainsys} with $f\equiv g\equiv1$. Although both ${u}$ and ${v}$ are positive on $\partial B_1$, the function ${v}$ is strictly negative in $B_{\epsilon\sqrt{2}}$. Observe that 
	$$
	\partial\{{u}>0\}=\{0\}, \quad \partial\{{v}> 0\}=\partial B_{\epsilon\sqrt{2}} \quad \mbox{and} \quad  \partial\{|({u},{v})|>0\}=\emptyset,
	$$
	and once again \eqref{set2} and \eqref{set3} fail.
    \end{example}	
    Our third example points out that estimate \eqref{growthest||} is optimal.
    \begin{example} In the previous example, by moving the paraboloid $v_\epsilon$ up and passing to the limit, as $\epsilon\to0$, we obtain 
    \begin{equation*}\label{ex3}
    u(x):=\dfrac{81}{64}|x|^{4/3} \quad \mbox{and} \quad v(x):=\frac{1}{2}|x|^2, \quad x \in B_1, 
    \end{equation*}
    which is a non-negative solution of \eqref{mainsys} with $f\equiv g\equiv1$, hence Theorem \ref{coupledFBs} holds for $( u,  v)$. In fact,
    $$
    \partial\{|( u, v)| > 0\}= \partial\{ u > 0\}= \partial\{ v > 0\}= \{0\}.  
    $$
    \end{example}
\begin{figure}[h!]
\centering
\begin{tikzpicture}
\begin{axis}[title=, hide axis, colormap/cool]
\addplot3[surf, domain=-3:3, samples=50]{y^2};
\addplot3[surf, domain=-3:3, samples=50]{(4)*((x^2+y^2)^(2/3))};
\addplot[magenta,mark=none,domain=0.1:3,line legend] {-0.1};
\addplot[magenta,mark=none,domain=-3:-0.05,line legend] {-0.1};
\end{axis}
\path (3.4,1.47) node[circle, fill, blue!40,inner sep=1]{};
\end{tikzpicture}
\caption{An illustration of a non-empty uncoupled free boundary $\partial\{v>0\}\setminus \partial\{u>0\}$ (see Example \ref{sinclusion}). The red line above is $\partial\{v>0\}$ and the blue point is $\partial\{u>0\}$. By Theorem \ref{singularset}, all points on $\partial\{v>0\}\setminus \partial\{u>0\}$ are singular. In this example, there are no regular points.}
\label{fig:2}
\end{figure}
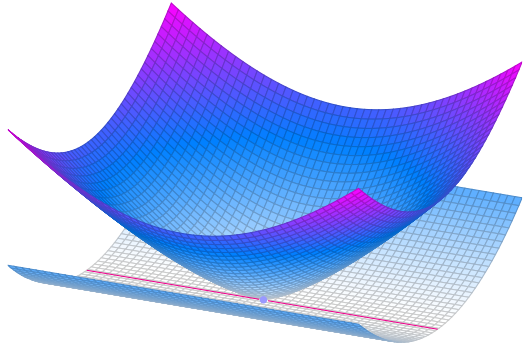

    The next example reveals that inclusion \eqref{set2} is strict, Remark \ref{strictinclusion}.
    
    \begin{example}\label{sinclusion}
      The pair of functions $(u,v)$, where
    \begin{equation*}\label{counter}
		u(x,y):=\frac{4}{3^\frac{4}{3}}(x^2+y^2)^{2/3} \quad \mbox{and} \quad  v(x,y):=\frac{1}{2}y^2,\quad x,y\in\R,
	\end{equation*}
     solves
    \begin{equation*}
    \begin{array}{rcl}
        \Delta_\infty u  = 1 & \mbox{in}  &  \mathbb{R}^2\setminus \{0\} \supseteq\left( \mathbb{R}^2 \setminus \{y=0\}\right) = \{v>0\},\\[0.2cm]
        \Delta v = 1 & \mbox{in}  &  \mathbb{R}^2 \supset \mathbb{R}^2 \setminus \{0\} = \{u>0\},
    \end{array}
    \end{equation*}
    while $\partial\{v>0\} \setminus \partial\{u>0\}=\{y=0\} \setminus \{0\}\neq \emptyset$. 
    \end{example}

\bigskip
	
\noindent{\bf Acknowledgments.} DJA thanks the Abdus Salam International Centre for Theoretical Physics (ICTP) for great hospitality during his research visits. DJA is partially supported by CNPq grants 310020/2022-0 and 420014/2023-3 and by grant 2019/0014 Paraiba State Research Foundation (FAPESQ). RT was partially supported by the King Abdullah University of Science and Technology (KAUST), by the Centre for Mathematics of the University of Coimbra (funded by the Portuguese Government through FCT/MCTES, DOI 10.54499/UIDB/00324/2020), and by FCT, DOI 10.54499/2022.02357.CEECIND/CP1714/CT0001.

	\end{document}